\documentclass[preprint,12pt]{elsarticle}
\usepackage{amsmath,amssymb,amsfonts,amsthm,graphicx,subfigure,moreverb,latexsym,algorithm}
\usepackage[all]{xy}
\usepackage{multirow}
\usepackage{mathtools}

\usepackage{multicol}

\usepackage{booktabs}

\usepackage{tabularx}
\newcolumntype{C}{>{\centering\arraybackslash}X}
\usepackage[below]{placeins} 
\usepackage[english]{babel}
\usepackage{graphicx}
\usepackage[skip=1ex]{caption}
\usepackage{tabularx}
\usepackage{hyperref}
\usepackage[utf8x]{inputenc}
\newtheorem{thm}{Theorem}[section]

\newtheorem{exa}[thm]{Example}
\newtheorem{pro}[thm]{Proposition}
\newtheorem{defn}[thm]{Definition}
\newtheorem{cor}[thm]{Corollary}
\newtheorem{rem}[thm]{Remark}

\makeatletter
\def\ps@pprintTitle{%
	\let\@oddhead\@empty
	\let\@evenhead\@empty
	\let\@oddfoot\@empty
	\let\@evenfoot\@oddfoot
}
\makeatother
\begin{document}
\begin{frontmatter}
\title{Parity alternating permutations starting with an odd integer}
\author[label1]{Frether Getachew Kebede\footnote{Corresponding author.}}
 \address[label1]{Department of Mathematics, College of Natural and Computational Sciences, Addis Ababa University, P.O.Box 1176, Addis Ababa, Ethiopia; e-mail: frigetach@gmail.com}
 \author[label2]{Fanja Rakotondrajao}
 \address[label2]{D\'epartement de  Math\'ematiques et Informatique, BP 907 Universit\'e d'Antananarivo, 101 Antananarivo, Madagascar; e-mail: frakoton@yahoo.fr}
\begin{abstract}
A Parity Alternating Permutation of the set $[n] = \{1, 2,\ldots, n\}$ is a permutation with even and odd entries alternatively. We deal with parity alternating permutations having an odd entry in the first position, PAPs. We  study the numbers that count the PAPs with even as well as odd parity. We also study a subclass of PAPs being derangements as well, Parity Alternating Derangements (PADs). Moreover, by considering the parity of these PADs we look into their statistical property of excedance. 
\end{abstract}

\begin{keyword}
parity \sep parity alternating permutation \sep parity alternating derangement \sep excedance\\
\MSC[2020] 05A05 \sep 05A15 \sep 05A19
\end{keyword}
\end{frontmatter}

\section{Introduction and preliminaries}
A permutation $\pi$ is a bijection from the set $[n] = \{1, 2,\ldots, n\}$ to itself and we will write it in standard representation as $\pi=\pi(1)\,\pi(2)\,\cdots\,\pi(n)$, or as the product of disjoint cycles. The parity of a permutation $\pi$ is defined as 
the parity of the number of transpositions (cycles of length two) in any representation of $\pi$ as a product of transpositions. One way of determining the parity of $\pi$ is by obtaining the sign of $(-1)^{n-c}$, where $c$ is the number of cycles in the cycle representation of $\pi$. That is, if the sign of  $\pi$ is -1, then $\pi$ is called an odd permutation, and an even permutation otherwise. 
For example, the permutation  $4\,2\,1\,7\,8\,6\,3\,5=(1\,\,4\,\,7\,\,3)(2)(5\,\, 8)(6)$, of length 8, is even since it has sign 1. 
All basic definitions and properties not explained here can be found in for example \cite{stanley2011enumerative} and \cite{bona2012combinatorics}.

According to \cite{tanimoto2010combinatorics}, a Parity Alternating Permutation over the set $[n]$ is a permutation, in standard form, with even and odd entries alternatively (in this general sense). The set $\mathcal{P}_n$ of all parity alternating permutations is a subgroup of the symmetric group $S_n$, the group of all permutations over $[n]$. The order of the set set $\mathcal{P}_n$ has been studied lately in relations to other number sequences such as Eulerian numbers (see \cite{tanimoto2010parity, tanimoto2010combinatorics}). 

However, in this paper we will deal only with the parity alternating permutations which in addition have an odd entry in the first position; and we call them PAPs. It can be shown that the set $P_n$ containing all PAPs over $[n]$ is a subgroup of the symmetric group $S_n$ and also of the group $\mathcal{P}_n$. We consider this kind of permutations because, for  odd $n$ there are no parity alternating permutations over $[n]$ beginning with an even integer. Avi Peretz determined the number sequence that count the number of PAPs (see \href{https://oeis.org/A010551}{A010551}). Unfortunately, we could not find any details of his work. In \href{https://oeis.org/A010551}{A010551}, we can also find the exponential generating function of these numbers due to Paul D. Hanna. Since there is no published proof of this formula we prove it here, as Theorem \ref{thmpigen}. Moreover, the numbers that count the PAPs with even parity and with odd parity (which were not studied before) are determined.

By $p_n$ we denote the cardinality of the set $P_n$ of all PAPs over $[n]$. Let $\phi_n$ denote a map from $P_n$ to $S_{\lceil\frac{n}{2}\rceil}\times S_{\lfloor\frac{n}{2}\rfloor}$ that relates a PAP $\sigma$ to a pair of permutations $(\sigma_1, \sigma_2)$ in the set $S_{\lceil\frac{n}{2}\rceil}\times S_{\lfloor\frac{n}{2}\rfloor}$, in such a way that $\sigma_1(i)=\frac{\sigma(2i-1)+1}{2}$ and $\sigma_2(i)=\frac{\sigma(2i)}{2}$. It is easy to see that this map is a bijection. 
For example, the PAPs $5\,2\,1\,4\,3\,6\,7$ and $7\,4\,5\,6\,3\,2\,1$ over [7] are mapped to the pairs $(3\,1\,2\,4,\,1\,2\,3)$ and $(4\,3\,2\,1,\, 2\,3\,1)$, respectively. If we consider a PAP $\sigma$ in cycle representations, then each cycle consists of integers of the same parity. Thus, we immediately get cycle representation of $\sigma_1$ and $\sigma_2$. For instance, the cycle form of the two PAPs above are $(1\,5\,3)(7)(2)(4)(6)$ and $(1\,7)(3\,5)(2\,4\,6)$ which correspond to the pairs $\left((1\,3\,2)(4),\,(1)(2)(3)\right)$ and 
$((1\,4)(2\,3),\,(1\,2\,3))$, respectively. (Unless stated otherwise we will always use (disjoint) cycle representation of permutations.) Another way of looking at the mapping $\phi_n$ is that $\sigma_1$ and $\sigma_2$ correspond to the parts that contain the odd and even integers in $\sigma$, respectively. Therefore, studying PAPs is similar to studying the two permutations that correspond to the even and the odd integers in the PAP separately and then combining the properties. In Table \ref{tab:t1}, we give a short summary of properties that permutations and PAPs satisfy (for detailed discussions, see Section \ref{section2}).
\begin{table}[H]
\centering
\begin{tabular}{ccc}
       & Permutations & PAPs\\
   \toprule
    Seq & $1, 1, 2, 6, 24, 120,\ldots.$ (\href{https://oeis.org/A000142}{A000142}) & $1, 1, 1, 2, 4, 12,\ldots.$ (\href{https://oeis.org/A010551}{A010551})\\[.2cm]
    EGF & $\frac{1}{1-x}$ & $\frac{2\sqrt{4-x^2}+2\cos^{-1}\left( 1-x^2/2\right)}{(2-x)\sqrt{4-x^2}}$\\[.2cm]
     Even (seq) & $1, 1, 1, 3, 12, 60, \ldots.$ (\href{https://oeis.org/A001710}{A001710}) & $1,1,1,1,2,6,18,72,\ldots$\\
    Odd (seq)& $1, 1, 1, 3, 12, 60, \ldots.$ (\href{https://oeis.org/A001710}{A001710}) & $0,0,0,1,2,6,18,72,\ldots$\\[.2cm]
    Even (EGF) & $\frac{2-x^2}{2-2x}$ & $\frac{\sqrt{4-x^2}+\cos^{-1}\left( 1-\frac{x^2}{2}\right)}{(2-x)\sqrt{4-x^2}}+\frac{x^2}{4}+\frac{x}{2}+\frac{1}{2}$\\[.3cm]
    Odd (EGF) & $\frac{x^2}{2-2x}$ & $\frac{\sqrt{4-x^2}+\cos^{-1}\left( 1-\frac{x^2}{2}\right)}{(2-x)\sqrt{4-x^2}}-\frac{x^2}{4}-\frac{x}{2}-\frac{1}{2}$\\
   \bottomrule
   \end{tabular}
   \caption{A comparison table of permutations and PAPs (EGF mean exponential generating function).}
   \label{tab:t1}
\end{table}
One interesting subset of $S_n$ is the set $D_n$ of derangements. For $d_n=|D_n|$, we have a well known relation 
\begin{align}\label{recd1}
    d_n=(n-1)[d_{n-1}+d_{n-2}],\,\,d_0=1 \text{ and } d_1=0
\end{align}
for $n\geq 2$. A proof of this relation may be found in any textbook on combinatorics, but we will have later use of the following bijection due to Mantaci and Rakotondrajao (\cite{mantaci2003exceedingly}). They define $\psi_n$ to be the bijection between $D_n$ and $[n-1]\times (D_{n-1}\cup D_{n-2})$ as follows:
let $D_n^{(1)}$ denote the set of derangements over $[n]$ having the integer $n$ in a cycle of length greater than 2, and  $D_n^{(2)}$ be the set of derangements over $[n]$ having $n$ in a transposition. These two sets are disjoint and their union is $D_n$. Then for $\delta\in D_n$ define $\psi_n(\delta)=(i,\delta')$, where $i=\delta^{-1}(n)$ and $\delta'$ is the derangement obtained from
\begin{enumerate}
    \item[$\bullet$] $\delta\in D_n^{(1)}$ by removing $n$ or
    \item[$\bullet$] $\delta\in D_n^{(2)}$ by removing the transposition $(i\,\,n)$ and then decreasing all integers greater than $i$ by 1.
\end{enumerate}
For instance, the pairs $(2,(1\,5\,2)(3\,4))$ and $(2,(1\,2)(3\,4))$ correspond to the derangements $(1\,5\,2\,6)(3\,4)$ and $(1\,3)(4\,5)(2\,6)$, respectively, for $n=6$. We denote the restricted bijections $\psi_n|_{D_n^{(1)}}$ and $\psi_n|_{D_n^{(2)}}$ by $\psi_n^{(1)}$ and $\psi_n^{(2)}$, respectively.

Another important, and more difficult to prove, recurrence relation that the numbers $d_n$ satisfy is 
\begin{align}\label{recd2}
    d_n=n\,d_{n-1}+(-1)^n,\,\,d_0=0
\end{align}
for $n\geq 1$. We will later make a use of the bijection $\tau_n: ([n]\times D_{n-1})\backslash F_n \longrightarrow D_n\backslash E_n$ given by the second author (\cite{rakotondrajao2007k}) proving the recurrence. Where $E_n$ is the set containing the derangement $\Delta_n=(1\,2) (3\,4)\cdots (n-1\,\,\,n)$ for even  $n$, and is empty for odd $n$. $F_n$ is the set containing the pair $(n,\,\Delta_{n-1})$ when $n$ is odd, and is empty when $n$ is even. Thus, the inverse $\zeta_n$ of $\tau_n$ relates an element of $[n-1]\times D_{n-1}$ with every derangement over $[n]$ that has the integer $n$ in a cycle of length greater than 2, and an element of $\{n\}\times D_{n-1}\backslash F_n$ with every derangement over $[n]$ in which $n$ lies in a transposition.

Classifying derangements by their parity, we denote the number of even and odd derangements  over $[n]$ by $d^e_n$ and $d^o_n$, respectively. Clearly $d_n=d^e_n+d^o_n$. Moreover, the numbers $d^e_n$ and $d^o_n$ satisfy the relations
\begin{align}\label{derpar}
    d^e_n=(n-1)[d^o_{n-1}+d^o_{n-2}]\quad \text{ and }\quad d^o_n=(n-1)[d^e_{n-1}+d^e_{n-2}],
\end{align}
for $n\geq 2$ with initial conditions $d^e_0=1$, $d^e_1=0$, $d^o_0=0$, and $d^o_1=0$ (\cite{mantaci2003exceedingly}, Proposition 4.1).

We will put a major interest on Parity Alternating Derangements (PADs) which are the derangements which also are parity alternating permutations starting with odd integers. Let $\mathfrak{d}_n$ denote cardinality of the set of PADs $\mathfrak{D}_n=D_n\cap P_n$. The restricted bijection $\Phi_n=\phi_n|_{\mathfrak{D}_n}:\mathfrak{D}_n\longrightarrow D_{\lceil\frac{n}{2}\rceil}\times D_{\lfloor\frac{n}{2}\rfloor}$ will let us consider the odd parts and the even parts of any given PAD regarded as ordinary derangements with smaller length than the length of the PAD. The mapping $\Phi_n$ plays the central role in our investigations. In Table \ref{tab:t2}, we display the connection of ordinary derangements and PADs (for detailed discussions, see Section \ref{section3}). Finding explicit expressions for some of the generating functions are still open questions. On the other hand the EGF for the PADs for example is the solution to an eighth order differential equation with polynomial coefficients, and also is expressible in terms of Hadamard products of some known generating functions.
\begin{table}[H]
\centering
\begin{tabular}{ccc}
   & Derangements & PADs\\
   \toprule
    Seq & $1, 0, 1, 2, 9, 44,\ldots$ (\href{https://oeis.org/A000166}{A000166}) & $1, 0, 0, 0, 1, 2, 4, 18, 81, 396, \ldots$ \\[.2cm]
    EGF & $\frac{e^{-x}}{1-x}$ & open\\[.2cm]
    RR & $d_n=(n-1)[d_{n-1}+d_{n-2}]$ & relation (\ref{recpad1}) \\
    RR & $d_n=n d_{n-1}+(-1)^n$ & relation (\ref{recpad2})\\
     Even (seq)& $1, 0, 0, 2, 3, 24, 130, \ldots$ (\href{https://oeis.org/A003221}{A003221}) & $1,0,0,0,1,0,4,6,45,192, 976\ldots$\\
    Odd (seq)& $0, 0, 1, 0, 6, 20, 135, \ldots$ (\href{https://oeis.org/A000387}{A000387}) & $0, 0, 0, 0, 0, 2, 0, 12, 36, 204, 960,\ldots$\\[.2cm]
    Even (EGF) & $\frac{(2-x^2) e^{-x}}{2(1-x)}$  & open\\[.2cm]
    Odd (EGF) & $\frac{x^2 e^{-x}}{2(1-x)}$ & open\\[.2cm]
     Even (RR) & $d^e_n=(n-1)[d^o_{n-1}+d^o_{n-2}]$ & relation (\ref{padpar1})\\
    Odd (RR) & $d^o_n=(n-1)[d^e_{n-1}+d^e_{n-2}]$ & relation (\ref{padpar2})\\[.2cm]
    Even\,-\,Odd & $(-1)^{n-1}(n-1)$ & $(-1)^{n-2}\Big\lceil \frac{n-2}{2}\Big\rceil \Big\lfloor \frac{n-2}{2}\Big\rfloor$\\
   \bottomrule
   \end{tabular}
   \caption{A comparison table of derangements and PAPs, RR represents recurrence relation.}
   \label{tab:t2}
\end{table}

In section \ref{section4}, we study excedance distribution over PADs by means of the corresponding distributions for the two derangements obtained by $\Phi_n$.

\section{Parity Alternating Permutations (PAPs)}\label{section2}
 As we stated in the introduction, we use splitting method by the mapping $\phi_n$ in the study of PAPs. One application of this is that the number of PAPs of length $n$ is
 \begin{align*}
 p_{n}=|S_{\lceil\frac{n}{2}\rceil}||S_{\lfloor\frac{n}{2}\rfloor}|=\lceil n/2\rceil !\lfloor n/2\rfloor !.
 \end{align*}
 
\begin{table}[ht!]
\centering
\begin{tabular}{c|ccccccccccc}
$n$ & 0 & 1 & 2 & 3 & 4 & 5 & 6 & 7 & 8 & 9 & 10\\ 
\hline
$p_n$ & 1 & 1 & 1 & 2 & 4 & 12 & 36 & 144 & 576 & 2880 & 14400\\ 
\end{tabular}
\caption{First few terms of the sequence $\{p_n\}_0^{\infty}$.}
\label{table_PAP}
\end{table}

  \begin{pro}\label{cor1}
 The numbers $p_n$ satisfy the recurrence relation
\begin{align*}\label{pi-n}
    p_{n}=\lceil n/2\rceil p_{n-1},
\end{align*}
for $n\geq 1$ and $p_0=1$.
\end{pro}
\begin{proof}
First let us define a mapping $\omega_n:S_n\longrightarrow [n]\times S_{n-1}$ by 
\begin{align*}
    \omega_n(\pi)=(i,\pi'),
\end{align*}
where $\pi'$ is obtained from $\pi\in S_n$ by removing the integer $n$, and $i=\pi^{-1}(n)$. One can easily see that $\omega_n$ is a bijection.
 
Now let us take a PAP $\sigma$ over $[n]$. Then $\phi_n$ maps $\sigma$ to a pair $(\sigma_1, \sigma_2)$. Define then a mapping $\Omega:P_n\longrightarrow \Bigl[ \big\lceil \frac{n}{2} \big\rceil \Bigr]\times P_{n-1}$ as follows: for $n=2m$
\begin{align*}
 \Omega(\sigma)=\left(i,\phi_{2m}^{-1}(\sigma_1,\, \sigma_2')\right),   
\end{align*}
where $(i,\sigma_2')=\omega_m(\sigma_2)$, and for $n=2m+1$
\begin{align*}
  \Omega(\sigma)=\left(i,\phi_{2m+1}^{-1}(\sigma_1',\, \sigma_2)\right),   
\end{align*}
where $(i,\sigma_1')=\omega_{m+1}(\sigma_1)$. The mapping $\Omega$ is a bijection since $\omega_n$ is a bijection for every $n\geq 1$. In any case, there are $\big\lceil \frac{n}{2} \big\rceil$ possibilities for $i$.
\end{proof}
As a consequence, we get the following theorem. 
\begin{thm}\label{thmpigen}
The exponential generating function $P(x)=\sum_{n\geq 0}p_n \frac{x^n}{n!}$ of the sequence $\{p_n\}_{n=0}^{\infty}$ has the closed formula
\[
P(x)=\displaystyle\frac{2}{2-x}+\displaystyle\frac{\cos^{-1}( 1-\displaystyle\frac{x^2}{2})}{(2-x)\sqrt{1-\displaystyle\frac{x^2}{4}}}.
\]
\end{thm}
\begin{proof}
Based on the recurrence relation in Proposition \ref{cor1}, we obtain the following relations
\begin{align*}
    &P_0(x)=\frac{x}{2}P_1(x)+1\quad \text{ and } \quad P_1(x)=\frac{x}{2}P_0(x)+\frac{1}{2}\int_{0}^{x}P_0(t)\,dt,
\end{align*}
where $P_0(x)=\sum_{n\geq 0}p_{2n}\frac{x^{2n}}{(2n)!}$ and $P_2(x)=\sum_{n\geq 0}p_{2n+1}\frac{x^{2n+1}}{(2n+1)!}$. Clearly, $P(x)=P_0(x)+P_1(x)$. Additionally, $P_0(x)$ satisfies the differential equation 
\begin{align*}
    \left(1-\frac{x^2}{4}\right)P'_0(x)=\frac{x^2+2}{2x}P_0(x)-\frac{1}{x}.
\end{align*}
Thus, we obtain the formulas
\begin{align*}
    P_0(x)=\frac{4}{4-x^2}+\frac{4x\sin^{-1}\left(\frac{x}{2}\right)}{(4-x^2)^{3/2}} \quad\text{ and }\quad P_1(x)=\frac{8}{4x-x^3}+\frac{8x\sin^{-1}\left(\frac{x}{2}\right)}{x(4-x^2)^{3/2}}-\frac{2}{x}.
\end{align*}
Therefore,
\[
P(x)=\displaystyle\frac{2}{2-x}+\displaystyle\frac{\cos^{-1}( 1-\displaystyle\frac{x^2}{2})}{(2-x)\sqrt{1-\displaystyle\frac{x^2}{4}}}.
\qedhere\]
\end{proof}
For classification of PAPs in terms of their parity, we use $P^e_n$ and $P^o_n$ to denote the set of even PAPs and odd PAPs, respectively, and $p_{n}^e$ and $p_{n}^o$ as their cardinality, respectively. Thus,
\begin{align*}
p_n=p^e_n+p^o_n.
\end{align*}
\begin{table}[ht!]
\centering
\begin{tabular}{c|ccccccccccc}
$n$ & 0 & 1 & 2 & 3 & 4 & 5 & 6 & 7 & 8 & 9 & 10\\ 
\hline 
$p_n^e$ & 1 & 1 & 1 & 1 & 2 & 6 & 18 & 72 & 288 & 1440 & 7200\\ 
$p_n^o$ & 0 & 0 & 0 & 1 & 2 & 6 & 18 & 72 & 288 & 1440 & 7200
\end{tabular}
\caption{First few terms of the sequences $\{p_n^e\}_0^{\infty}$ and $\{p_n^o\}_0^{\infty}$.}
\label{table_PAP_e_o}
\end{table}

Our goal is now to study the relationships between these two sequences.
\begin{thm}\label{prop1}
The numbers $p_n^e$ and $p_n^o$ satisfy the recurrence relations
\begin{align*}
&p_{n}^e=\lfloor (n-1)/2\rfloor p_{n-1}^o+p_{n-1}^e\\
&p_{n}^o=\lfloor (n-1)/2\rfloor p_{n-1}^e+p_{n-1}^o,
\end{align*}
for $n\geq 1$, with initial conditions $p^e_0=1$ and  $p^o_0=0$. 
\end{thm}

\begin{proof}
Let $S_n^e$ and $S_n^{o}$ be the set of even and odd permutations, respectively. Define two mappings $\omega_n^{e}:S_n^e\longrightarrow [n-1]\times S_{n-1}^{o}\cup S_{n-1}^e$ and $\omega_n^{o}:S_n^o\longrightarrow [n-1]\times S_{n-1}^{e}\cup S_{n-1}^o$ by
\begin{align*}
    \omega_n^{e}(\pi)=\begin{cases}
                     (i,\pi'), \text{ if } i\neq n\\
                     \pi'', \text{ otherwise}
                     \end{cases}
                     \quad \text{ and } \quad 
    \omega_n^{o}(\pi)=\begin{cases}
                     (j,\pi'), \text{ if } j\neq n\\
                     \pi'', \text{ otherwise}
                     \end{cases},
\end{align*}
respectively, where $i=\pi^{-1}(n)$, $\pi'$ is obtained from $\pi$ by removing the integer $n$, and $\pi''$ is obtained from $\pi$ by removing the cycle $(n)$, for $\pi\in S_n^e$. Similarly for $\omega_n^o$. It is easy to see that both mappings $\omega_n^{e}$ and $\omega_n^{o}$ are bijections.

The mapping $\omega_n^{e}$ changes the parity of $\pi$ when it results in $\pi'$ and preserves when it results in $\pi''$. This is because the signs of $\pi$, $\pi'$ and $\pi''$ are $(-1)^{n-c}$, $(-1)^{n-1-c}$, and $(-1)^{n-1-c+1}$, respectively, where $c$ is the number of cycles in $\pi$. For the mapping $\omega_n^o$ we apply similar argument. 

Now consider a PAP $\sigma$ in $P_n$. Then $\phi_n$ maps $\sigma$ in to a pair $(\sigma_1, \sigma_2)$. Following the notation in the proof of Proposition \ref{cor1}, let us define two mappings $\Omega^e:P_n^e\longrightarrow\Bigl[\big\lfloor \frac{n-1}{2}\big\rfloor\Bigr]\times P_{n-1}^{o}\cup P_{n-1}^{e}$ and $\Omega^o:P_n^o\longrightarrow\Bigl[\big\lfloor \frac{n-1}{2}\big\rfloor\Bigr]\times P_{n-1}^{e}\cup P_{n-1}^{o}$ as follows:
\begin{enumerate}
    \item when $n$ is even
    \begin{align*}
    &\Omega^e(\sigma)=\begin{cases} \left(i,\,\phi_{n}^{-1}(\sigma_1, \sigma_2')\right), \text{ if } i\neq \frac{n}{2}\\
    \phi_{n}^{-1}(\sigma_1, \sigma_2''), \text{ otherwise}
    \end{cases}\text{and }\\
    &\Omega^o(\sigma)=\begin{cases} \left(i,\,\phi_{n}^{-1}(\sigma_1, \sigma_2')\right), \text{ if } i\neq \frac{n}{2}\\
    \phi_{n}^{-1}(\sigma_1, \sigma_2''), \text{ otherwise}\end{cases},
    \end{align*}
    where $i=\sigma_2^{-1}(\frac{n}{2})$, and both $\sigma'_2$,  $\sigma''_2$  are obtained from  $\sigma_2$ by the mapping $\omega_{\frac{n}{2}}^{e}$ when $\sigma\in P_n^e$ and by the mapping  $\omega_{\frac{n}{2}}^{o}$ when $\sigma\in P_n^o$,
     \item when $n$ is odd
\begin{align*}
    &\Omega^e(\sigma)=\begin{cases} \left(j,\,\phi_{n}^{-1}(\sigma_1', \sigma_2)\right), \text{ if } j\neq \frac{n+1}{2}\\
    \phi_{n}^{-1}(\sigma_1'', \sigma_2), \text{ otherwise} \end{cases}
\text{and }\\
    &\Omega^o(\sigma)=\begin{cases} \left(j,\,\phi_{n}^{-1}(\sigma_1', \sigma_2)\right), \text{ if } j\neq \frac{n+1}{2}\\
    \phi_{n}^{-1}(\sigma_1'', \sigma_2), \text{ otherwise}
    \end{cases}, 
    \end{align*}
    where $j=\sigma_1^{-1}(\frac{n+1}{2})$, and both $\sigma'_1$, $\sigma''_1$  are obtained from  $\sigma_1$ by the mapping $\omega_{\frac{n+1}{2}}^{e}$ when $\sigma\in P_n^e$ and by the mapping $\omega_{\frac{n+1}{2}}^{o}$ when $\sigma\in P_n^o$.
\end{enumerate} 
Since $\omega_n$ is bijection for $n\geq 2$, both $\Omega^e$ and $\Omega^o$ are bijections too. Note that in both mappings there are  $\big\lfloor \frac{n-1}{2} \big\rfloor$ possibilities for $i$ ($i\neq \frac{n}{2}$) and similarly for $j$ ($j\neq \frac{n+1}{2}$).
\end{proof}

\begin{pro}\label{propequal}
For any positive integer $n\geq 3$, we have
\begin{align*}
p_n^e=p_n^o.
\end{align*}
\end{pro}

\begin{proof}
 Multiplying a PAP by a transposition $(1,n)$ if $n$ is odd, or by $(1,n-1)$ if $n$ is even, we obtain a PAP having opposite parity. This multiplication means swapping the first and the last odd integer of a PAP in standard representation. It creates a bijection between $P^e_n$ and $P^o_n$.
\end{proof}

By applying Proposition \ref{propequal} and considering $p(x)$, we get:

\begin{cor}\label{cor2}
The exponential generating functions of the sequences $\{p_n^e\}_{n\geq 0}$ and $\{p_n^o\}_{n\geq 0}$ have the closed forms
\begin{align*}
P^e(x)=\displaystyle\frac{1}{2}\left(P(x)+\frac{x^2}{2}+x+1\right)\quad \text{ and }\quad P^o(x)=\displaystyle\frac{1}{2}\left(P(x)-\frac{x^2}{2}-x-1\right).
\end{align*}
\hfill$\square$
\end{cor}

\section{Parity Alternating Derangements (PADs)}\label{section3}
As a result of the bijection $\Phi_n$ in the introduction above, we can determine the number $\mathfrak{d}_n$ of PADs over $[n]$ as follows:
\begin{align*}
    \mathfrak{d}_{n}=d_{\lceil n/2\rceil}d_{\lfloor n/2\rfloor}=\sum_{j=0}^{\lceil n/2\rceil}\sum_{i=0}^{\lfloor n/2\rfloor} \lceil n/2\rceil ! \lfloor n/2\rfloor !\frac{(-1)^{i+j}}{j!\,i!}.
\end{align*}

\begin{table}[ht!]
\centering
\begin{tabular}{c|cccccccccc}
$n$ & 0 & 1 & 2 & 3 & 4 & 5 & 6 & 7 & 8 & 9 \\ 
\hline 
$d_n$ & 1 & 0 & 1 & 2 & 9 & 44 & 265 & 1854 & 14833 & 133496 \\ 
$\mathfrak{d}_n$ & 1 & 0 & 0 & 0 & 1 & 2 & 4 & 18 & 81 & 396  \\ 
\end{tabular} 
\caption{First few values of $d_n$ and $\mathfrak{d}_n$.}
\label{table_PAD_n}
\end{table}

In the next theorem we give a formula for the number of PADs, connected to the relation (\ref{recd1}).
\begin{thm}\label{thmpad1}
The number of PADs over $[n]$ satisfy the recurrence relation
\begin{align}\label{recpad1}
    \mathfrak{d}_{n}=s\big(\mathfrak{d}_{n-1}+(n-2-s)\left(\mathfrak{d}_{n-3}+\mathfrak{d}_{n-4}\right)\big),
\end{align}
where $s=\big\lfloor\frac{n-1}{2} \big\rfloor=\frac{2n-3-(-1)^n}{4}$, for $n\geq 4$, with initial conditions $\mathfrak{d}_0=1$, $\mathfrak{d}_1=0$, $\mathfrak{d}_2=0$ and $\mathfrak{d}_3=0$.
\end{thm}

\begin{proof}
The proof is splitted into two cases, for PADs over a set of odd and even sizes. Let $(\delta_1, \delta_2)$ be the corresponding  pair of a PAD $\delta\in \mathfrak{D}_n$ under the mapping $\Phi_n$. Define a mapping $\Psi: \mathfrak{D}_{n} \longrightarrow [s]\times\left(\mathfrak{D}_{n-1}\cup [n-2-s]\times (\mathfrak{D}_{n-3}\cup \mathfrak{D}_{n-4})\right)
$ as follows:

\textbf{Case I:} for odd $n$, depending on the following two elective properties of $\delta$, the mapping $\Psi$ will be defined as:
\begin{enumerate}
    \item in the event of the largest entry $\frac{n+1}{2}$ of $\delta_1$ being in a cycle of length greater than 2, we let
    \begin{align*}
        \Psi(\delta)=\left(i,\,\Phi_{n}^{-1} (\delta'_1,\, \delta_2 )\right),
    \end{align*}
    where $(i,\delta'_1)=\psi_{\frac{n+1}{2}}^{(1)}(\delta_1)$;
   \item in the event when $\frac{n+1}{2}$ lies in a transposition in $\delta_1$, we distinguish two cases:
   \begin{itemize}
       \item if the largest entry $\frac{n-1}{2}$ of $\delta_2$ is contained in a cycle of length greater than 2, then
       \begin{align*}
        \Psi(\delta)=\left(i,\,j,\,\Phi_{n}^{-1} (\delta'_1,\, \delta'_2)\right),
    \end{align*}
    where $(i,\,\delta'_1)=\psi_{\frac{n+1}{2}}^{(2)}(\delta_1)$ and $(j,\,\delta'_2)=\psi_{\frac{n-1}{2}}^{(1)}(\delta_2)$;
    \item if $\frac{n-1}{2}$ is contained in a cycle of length 2 in $\delta_2$, then
    \begin{align*}
        \Psi(\delta)=\left(i,\,j,\,\Phi_{n}^{-1} (\delta'_1,\, \delta'_2)\right),
    \end{align*}
    where $(i,\,\delta'_1)=\psi_{\frac{n+1}{2}}^{(2)}(\delta_1)$ and $(j,\,\delta'_2)=\psi_{\frac{n-1}{2}}^{(2)}(\delta_2)$.
   \end{itemize}
\end{enumerate}
\textbf{Case II:} for even $n$,
\begin{enumerate}
    \item in the event of the largest entry $\frac{n}{2}$ of $\delta_2$ lies in a cycle of length greater than 2, we let
    \begin{align*}
        \Psi(\delta)=\left(i,\,\Phi_{n}^{-1} (\delta_1,\, \delta'_2)\right),
    \end{align*}
    where $(i,\,\delta'_2)=\psi_{\frac{n}{2}}^{(1)}(\delta_2)$;
   \item in the event when $\frac{n}{2}$ being in a transposition in $\delta_2$, we distinguish two cases:
   \begin{itemize}
       \item if the largest entry $\frac{n}{2}$ in $\delta_1$ contained in a cycle of length greater than 2, then
       \begin{align*}
        \Psi(\delta)=\left(i,\,j,\,\Phi_{n}^{-1} (\delta'_1,\, \delta'_2)\right),
    \end{align*}
    where $(i,\,\delta'_1)=\psi_{\frac{n}{2}}^{(1)}(\delta_1)$ and $(j,\,\delta'_2)=\psi_{\frac{n}{2}}^{(2)}(\delta_2)$;
    \item if $\frac{n}{2}$ contained in a cycle of length 2 in $\delta_1$, then
    \begin{align*}
        \Psi(\delta)=\left(i,\,j,\,\Phi_{2n}^{-1} (\delta'_1,\, \delta'_2)\right),
    \end{align*}
    where $(i,\,\delta'_1)=\psi_{\frac{n}{2}}^{(2)}(\delta_1)$ and $(j,\,\delta'_2)=\psi_{\frac{n}{2}}^{(2)}(\delta_2)$.
   \end{itemize}
\end{enumerate}
Since $\psi_n$ is a bijection for any $n\geq 2$, one can easily conclude that $\Psi$ is a bijection too. Note that in both cases there are $\big\lfloor\frac{n-1}{2}\big\rfloor=s$ possibilities for $i$ and $\big\lfloor\frac{n-2}{2}\big\rfloor=n-2-s$ possibilities for $j$. Thus, the formula in the Theorem follows.
\end{proof}

The next theorem is connected to the relation (\ref{recd2}).

\begin{thm}
The number $\mathfrak{d}_{n}$ of PADs also satisfies the relation
\begin{align}\label{recpad2}
    \mathfrak{d}_{n}=s\mathfrak{d}_{n-1}+(-1)^{s}d_{n-s},
\end{align}
where $s=\big\lceil\frac{n}{2}\big\rceil=\frac{2n+1-(-1)^n}{4}$, for $n\geq 1$ with $d_1=0$, $d_0=1$ and $\mathfrak{d}_0=1$.
\end{thm}

\begin{proof}
Distinguishing by means of the parity of $n$, we can write the relation (\ref{recpad2}) as:
 \begin{align*}
     &\mathfrak{d}_{n}=d_{s}\big(s\, d_{s-1}+(-1)^{s}\big)\,\,\text{ when } n \text{ is even,  and }\\ &\mathfrak{d}_{n}=\big(s\,d_{s-1}+(-1)^{s}\big)\,d_{s-1}\,\,\text{ when } n \text{ is odd}.
 \end{align*} 
 Now, take a PAD $\delta$ in $\mathfrak{D}_{n}$ and introduce two mappings as: 
 \begin{itemize}
     \item  $Z_{0}: \big(D_s\backslash E_s\big)\times D_s \longrightarrow [s]\times \big(D_{s-1}\backslash F_s\big)\times D_{s}$ by 
 \begin{align*}
     \delta \xrightleftharpoons[\Phi^{-1}_{n}]{\Phi_{n}}(\delta_1, \delta_2)\xrightleftharpoons[id_s\times \tau_s]{ id_s\times \zeta_s}\big(\delta_1, (i,\delta'_2)\big)\xrightleftharpoons[h]{h}(i,(\delta_1, \delta'_2))\xrightleftharpoons[id\times \Phi_n]{id\times \Phi_n^{-1}}\big(i,\Phi_{n}^{-1}(\delta_1, \delta'_2)\big),
 \end{align*}
 and 
\item $Z_{1}: \big(D_{s}\backslash E_{s}\big)\times D_{s-1} \longrightarrow [s]\times \big(D_{s-1}\backslash F_{s}\big)\times D_{s-1}$ by 
 \begin{align*}
     \delta \xrightleftharpoons[\Phi^{-1}_{n}]{\Phi_{n}}(\delta_1, \delta_2)\xrightleftharpoons[\tau_{s}\times id_{s-1}]{\zeta_{s}\times id_{s-1}}\big((i,\delta'_1),\delta_2\big)\xrightleftharpoons[h_2]{h_1}\big(i,(\delta'_1, \delta_2)\big)\xrightleftharpoons[id\times \Phi_n^{-1}]{id\times \Phi_n}\big(i,\Phi_{n}^{-1}(\delta'_1, \delta_2)\big).
 \end{align*}

\end{itemize}
Note that $h$, $h_1$, and $h_2$ are the obvious recombination maps. Since all the functions we used to define the two mappings $Z_{0}$ and $Z_{1}$ are injective, both $Z_{0}$ and $Z_{1}$ are bijection mappings.
\end{proof}

In order to classify PADs with respect to their parity, we let $\mathfrak{D}_n^e$ and $\mathfrak{D}_n^o$ denote the set of even and odd PADs over $[n]$, respectively. Moreover, $\mathfrak{d}^e_n=|\mathfrak{D}_n^e|$ and $\mathfrak{d}^o_n=|\mathfrak{D}_n^o|$. Obviously, $\mathfrak{d}_n=\mathfrak{d}^e_n+\mathfrak{d}^o_n$.

\begin{table}[ht!]
\centering
\begin{tabular}{c|ccccccccccc}
$n$ & 0 & 1 & 2 & 3 & 4 & 5 & 6 & 7 & 8 & 9 & 10 \\ 
\hline 
$\mathfrak{d}_n^e$ & 1 & 0 & 0 & 0 & 1 & 0 & 4 & 6 & 45 & 192 & 976\\ 
$\mathfrak{d}_n^o$ & 0 & 0 & 0 & 0 & 0 & 2 & 0 & 12 & 36 & 204 & 960
\end{tabular}
\caption{First few values of the number of even and odd PADs.}
\label{table_eopads}
\end{table}

\begin{pro}\label{thm3}
The numbers of even and odd PADs satisfy the relations
\begin{align*}
    &\mathfrak{d}^e_{n}=d^e_{\lfloor\frac{n}{2}\rfloor}d^e_{\lceil\frac{n}{2}\rceil}+d^o_{\lfloor\frac{n}{2}\rfloor}d^o_{\lceil\frac{n}{2}\rceil}, \\
    &\mathfrak{d}^o_{n}=d^e_{\lfloor\frac{n}{2}\rfloor} d^o_{\lceil\frac{n}{2}\rceil}+d^e_{\lceil\frac{n}{2}\rceil}d^o_{\lfloor\frac{n}{2}\rfloor},
\end{align*}
for $n\geq 0$, with initial conditions $d^e_0=1$, $d^e_1=0$, $d^o_0=0$, and $d^o_1=0$.
\end{pro}

\begin{proof}
Let $\delta$ be a PAD over $[n]$. Then, there exist $\delta_1\in D_{\lceil\frac{n}{2}\rceil}$ and $\delta_2\in D_{\lfloor\frac{n}{2}\rfloor}$ such that $\Phi_n(\delta)=(\delta_1, \delta_2)$. If $\delta\in \mathfrak{D}_n^e$, then $\delta_1$ and $\delta_2$ must have the same parity. Thus, $\mathfrak{d}^e_{n}=d^e_{\lfloor\frac{n}{2}\rfloor}d^e_{\lceil\frac{n}{2}\rceil}+d^o_{\lfloor\frac{n}{2}\rfloor}d^o_{\lceil\frac{n}{2}\rceil}$. If $\delta\in \mathfrak{D}_n^e$, then $\delta_1$ and $\delta_2$ must have opposite parities. Hence, $\mathfrak{d}^o_{n}=d^e_{\lfloor\frac{n}{2}\rfloor} d^o_{\lceil\frac{n}{2}\rceil}+d^e_{\lceil\frac{n}{2}\rceil}d^o_{\lfloor\frac{n}{2}\rfloor}$.
\end{proof}

\begin{cor}\label{rec_pad-parity}The number of PADs with even parity and with odd parity satisfy the recurrence relations
\begin{align}
    &\mathfrak{d}_{n}^e=s\left(\mathfrak{d}_{n-1}^o+(n-2-s)(\mathfrak{d}_{n-3}^e+\mathfrak{d}_{n-4}^e)\right)\label{padpar1}\\
    &\mathfrak{d}_{n}^o=s\left(\mathfrak{d}_{n-1}^e+(n-2-s)(\mathfrak{d}_{n-3}^o+\mathfrak{d}_{n-4}^o)\right),\label{padpar2}
\end{align}
where $s=\big\lfloor\frac{n-1}{2} \big\rfloor=\frac{2n-3-(-1)^n}{4}$, for $n\geq 4$ with initial conditions $\mathfrak{d}_0^e=1$, $\mathfrak{d}_0^o=0$, and $\mathfrak{d}_i^e=\mathfrak{d}_i^o=0$ for $i=1,2,3$.
\end{cor}
\begin{proof}
It is enough to clarify the effect of the bijection $\psi_n$ on the parity of a derangement over $[n]$, the rest is just applying the bijection $\Psi$ from the proof of Theorem \ref{thmpad1}. 

Letting $\delta$ be in $D_n$, $\delta'$ has sign either $(-1)^{n-1-c}$ if $\delta\in D_n^{(1)}$, or $(-1)^{n-2-(c-1)}=(-1)^{n-1-c}$  if $\delta\in D_n^{(2)}$. Here $\delta'$ is the derangement obtained from $\delta$ by applying $\psi_n$, and $c$ is the number of cycles in the cycle representation of $\delta$. This means, the bijection $\psi_n$ changes the parity of a derangement.
\end{proof}

\begin{defn}
Let $\delta$ be a derangement over $[n]$ in standard cycle representation and let $C_1=(1 \,\, a_2 \,\, \cdots \,\, a_m)$ be the first cycle. Following \cite{benjamin2005recounting}, an extraction point of $\delta$ is an entry $e\geq 2$ if $e$ is the smallest number in the set $\{2,\ldots,n\}\backslash\{a_2\}$ for which $C_1$ does not end with the numbers of $\{2,\ldots,e\}\backslash\{a_2\}$ written in decreasing order. The $(n-1)$ derangements, $\delta_{n,i}=(1\,\, i \,\, n \,\, n-1\,\, \cdots \,\, i+2 \,\, i+1 \,\, i-1 \,\, i-2 \,\, \cdots \,\, 3 \,\, 2)$ for $i\in[2, n]$, that do not have extraction points are called the exceptional derangements and the set of exceptional derangements is denoted by  $X_n$. Note that the extraction point (if it exists) must belong to the first or the second cycle.
\end{defn}
Following this approach we may introduce:
\begin{defn} 
We call the PAD 
\begin{align*}
    \Phi_n^{-1}(\delta_{\lceil\frac{n}{2}\rceil,i},\,\delta_{\lfloor\frac{n}{2}\rfloor,j}), \text{ for } i\in\big[2, \lceil n/2\rceil\big] \text{ and } j\in\big[2, \lfloor n/2\rfloor\big]
\end{align*}
 an exceptional PAD and we let $\mathcal{X}_n$ denote the set containing them. 
\end{defn}

\begin{exa}
If $n=8$, then 
\begin{align*}
  \mathcal{X}_8&=\{\Phi_8^{-1}(\delta_{4,2},\,\delta_{4,2}),\,\,\Phi_8^{-1}(\delta_{4,2},\,\delta_{4,3}),\,\,\Phi_8^{-1}(\delta_{4,2},\,\delta_{4,4}),\,\,\Phi_8^{-1}(\delta_{4,3},\,\delta_{4,2}),\,\,\Phi_8^{-1}(\delta_{4,3},\,\delta_{4,3}),\\
  &\quad\quad\Phi_8^{-1}(\delta_{4,3},\,\delta_{4,4}),\,\,\Phi_8^{-1}(\delta_{4,4},\,\delta_{4,2}),\,\,\Phi_8^{-1}(\delta_{4,4},\,\delta_{4,3}),\,\,\Phi_8^{-1}(\delta_{4,4},\,\delta_{4,4})\}\\
   &=\{(1\,3\,7\,5)(2\,4\,8\,6),\,\,(1\,3\,7\,5)(2\,6\,8\,4),\,\,(1\,3\,7\,5)(2\,8\,6\,4),\,\,(1\,5\,7\,3)(2\,4\,8\,6),\,\,(1\,5\,7\,3)\\&\quad\quad(2\,6\,8\,4),\,\,(1\,5\,7\,3)(2\,8\,6\,4),\,\,(1\,7\,5\,3)(2\,4\,8\,6),\,\,(1\,7\,5\,3)(2\,6\,8\,4),\,\,(1\,7\,5\,3)(2\,8\,6\,4)\}.
\end{align*}
\end{exa}
\begin{rem}
Since the exceptional derangements over $[n]$ have sign $(-1)^{n-1}$, all the exceptional PADs in $\mathcal{X}_n$ have the same parity, with sign $(-1)^{n-2}=(-1)^{n}$.
\end{rem}
\begin{rem}
As it was proved in \cite{benjamin2005recounting}, the number of the exceptional derangements in $X_n$ is the difference of the number of even and odd derangements, i.e., $d_n^e-d_n^o=(-1)^{n-1}(n-1)$. Chapman (\cite{chapman2001involution}) also provide a bijective proof for the same formula. Below we give the idea of the  proof due to Benjamin, Bennett,  and  Newberger (\cite{benjamin2005recounting}).

Let $f_n$ be the involution on $ D_n\backslash X_n$ defined by 
\begin{align*}
    f_n(\pi)=f_n\big((1\,\,a_2\,\,X\,\,e\,\,Y\,\,Z)\,\pi'\big)=(1\,\,a_2\,\,X\,\,Z)(e\,\,Y)\,\pi'
\end{align*}
for $\pi$ in $D_n\backslash X_n$ with the extraction point $e$ in the first cycle; and vice versa for the other $\pi$ in $D_n\backslash X_n$ with the extraction point $e$ in the second cycle. $a_2$ is the second element in the first cycle of $\pi$; $X$, $Y$, and $Z$ are ordered subsets of $[n]$, $Y\neq \emptyset$ and $Z$ consist the elements of $\{2,3,\ldots,e-1\}\backslash \{a_2\}$ written in decreasing order, and $\pi'$ is the rest of the derangement in $\pi$. Since the number of the cycles in $\pi$ and $f_n(\pi)$ differ by one, they must have opposite parity.
\end{rem}

Labeling $\mathfrak{d}_n^e{-}\mathfrak{d}_n^o$ as $\mathfrak{f}_n$, we have the following result.

\begin{pro}\label{diff_eg}
The difference $\mathfrak{f}_n$ counts the number of exceptional  PADs over $[n]$ and its closed formula is given by
\begin{align}\label{paddif}
\mathfrak{f}_n=(-1)^{n-2}\Big\lceil \frac{n-2}{2}\Big\rceil \Big\lfloor \frac{n-2}{2}\Big\rfloor.
\end{align}
\end{pro}

\begin{proof}
Let $\delta$ be in $\mathfrak{D}_n\backslash\mathcal{X}_n$. Then $\Phi_n$ map $\delta$ with the pair $(\delta_1,\delta_2)$.
Define a mapping $F$ from $\mathfrak{D}_n\backslash\mathcal{X}_n$ to itself as
\begin{align*}
    F(\delta)=\begin{cases}\Phi_n^{-1}\left(f_{\lceil\frac{n}{2}\rceil}(\delta_{1}),\,\delta_{2}\right) \text{ if } n \text{ is odd}\\
    \Phi_n^{-1}\left(\delta_{1},\,f_{\lfloor\frac{n}{2}\rfloor}(\delta_{2})\right) \text{ otherwise} \end{cases}
\end{align*}
Since $f_n$ is a bijection and changes  parity, $F$ is a bijection and also $\delta$ and $F(\delta)$ have opposite parity. The leftovers, which are the PADs in $\mathcal{X}_n$ with sign $(-1)^{n-2}$, are counted by $\lceil \frac{n-2}{2}\rceil\lfloor \frac{n-2}{2}\rfloor$.
\end{proof}

\begin{table}[ht!]
\centering
\begin{tabular}{c|ccccccccccccc}
$n$ & 0 & 1 & 2 & 3 & 4 & 5 & 6 & 7 & 8 & 9 & 10 & 11 & 12 \\ 
\hline 
$\mathfrak{f}_n$ & 1 & 0 & 0 & 0 & 1 & -2 & 4 & -6 & 9 & -12 & 16 & -20 & 25
\end{tabular}
\caption{First few values the difference $\mathfrak{f}_n=\mathfrak{d}_n^e-\mathfrak{d}_n^o$.}
\label{table_diffpad}
\end{table}

This sequence looks like an alternating version of \href{https://oeis.org/A002620}{A002620}, to which Paul Barry constructed an EGF. From \cite{paulbarry}, we learned that he used Mathematica to generate the formulas by taking the Inverse Laplace Transform  of the ordinary generating function described in \cite{barry2020central}. However, we propose the following more direct, constructive proof.

\begin{thm} 
The exponential generating function of the difference $\mathfrak{f}_n$ has the closed form
\begin{align*}
    \frac{e^x}{8}+\frac{e^{-x}}{8}(2x^2+6x+7).
\end{align*}
\end{thm}

\begin{proof}
From the closed formula of $\mathfrak{f}_n$ in Proposition \ref{diff_eg}, we have  
\begin{align*}
(n-1)^2=\mathfrak{f}_{2n}&=\mathfrak{d}_{2n}^e-\mathfrak{d}_{2n}^o\\
n(n-1)=\mathfrak{f}_{2n+1}&=-(\mathfrak{d}_{2n+1}^e-\mathfrak{d}_{2n+1}^o).
\end{align*}
Hence,
\begin{align*}
\sum_{n\geq 0}\mathfrak{f}_{2n}\frac{x^n}{(2n)!}=\sum_{n\geq 0}(n-1)^2\frac{x^{2n}}{(2n)!}=\frac{x^2-3x+4}{8}e^x+\frac{x^2+3x+4}{8}e^{-x},\quad\text{ and }\\
\sum_{n\geq 0}\mathfrak{f}_{2n+1}\frac{x^{2n+1}}{(2n+1)!}=-\sum_{n\geq 0}n(n-1)\frac{x^n}{n!}=-\frac{x^2-3x+3}{8}e^x+\frac{x^2+3x+3}{8}e^{-x}
\end{align*}
Thus,
\begin{align*}
\sum_{n\geq 0}\mathfrak{f}_{2n}\frac{x^{2n}}{(2n)!}+\sum_{n\geq 0}\mathfrak{f}_{2n+1}\frac{x^{2n+1}}{(2n+1)!}&=\frac{e^x}{8}+\frac{e^{-x}}{8}(2x^2+6x+7)
\end{align*}
is the desired formula.
\end{proof}

\section{Excedance distribution over PADs}\label{section4}
In this section, we focus on excedance distribution in PADs.

\begin{defn}
We say that  a permutation $\sigma$  has an excedance on $i\in[n]$ if $\sigma(i)>i$. In this case, $i$ is said to be an excedant.
\end{defn}

We give the notation below in the study of this property:
\begin{align*}
\mathfrak{d}_{n,k}&=\vert \{\delta\in \mathfrak{D}_n:\,\,\delta\text{ has }k\text{ excedances}\}\vert,\\
\mathfrak{d}_{n,k}^o&=\vert\{ \delta\in \mathfrak{D}^o(n):\,\,\delta \text{ has }k\text{ excedances}\}\vert,\\
\mathfrak{d}_{n,k}^e&=\vert\{ \delta\in \mathfrak{D}^e(n):\,\,\delta \text{ has }k\text{ excedances}\}\vert .
\end{align*}

 Mantaci and Rakotondrajao (\cite{mantaci2003exceedingly}) studied the excedance distribution in derangements, i.e., the numbers
 \begin{align*}
 d_{n,k}&=\vert \{\delta\in D_n:\,\,\delta\text{ has }k\text{ excedances}\}\vert,\\
d^e_{n,k}&=\vert\{ \delta\in D_n:\,\,\delta \text{ is an even derangemnt and } \text{ has }k\text{ excedances}\}\vert,\\
d^o_{n,k}&=\vert\{ \delta\in D_n:\,\,\delta \text{ is an odd derangemnt and } \text{ has }k\text{ excedances}\}\vert.
\end{align*}

\begin{rem}
Since the number of excedances in a derangement over $[n]$ is in the range $[1,n-1]$, the number of excedances of a PAD over $[n]$ is at least 2 and at most $n-2$.
\end{rem}

\begin{pro}\label{sym}
The numbers $\mathfrak{d}_{n,k}$, $\mathfrak{d}_{n,k}^e$, and $\mathfrak{d}_{n,k}^o$ are symmetric, that is
\begin{align*}
    \mathfrak{d}_{n,k}=\mathfrak{d}_{n,n-k},\,\, \mathfrak{d}^e_{n,k}=\mathfrak{d}^e_{n,n-k},\text{ and }\,\,\mathfrak{d}^o_{n,k}=\mathfrak{d}^o_{n,n-k}.
\end{align*}
\end{pro}

\begin{proof}
The bijection from $\mathfrak{D}_n$ to it self, defined as $\delta\mapsto\delta^{-1}$ for $\delta\in \mathfrak{D}_n$, associates a PAD having $k$ excedances with a PAD having $n-k$ exeedances and also preserves parity.
\end{proof}

\begin{table}[ht!]
        \centering
        \begin{tabular}{c|ccccccc}
\hline
 &  &  & $\mathfrak{d}_{n,k}$ &  &   & \\
 \hline
$n\setminus k$ & 2 & 3 & 4 & 5 & 6 & 7\\
\hline  
4 & 1 &  &  & &  &\\  
5 & 1 & 1 &  & & &  \\  
6 & 1 & 2 & 1 &  & &\\ 
7 & 1 & 8 & 8 & 1 & & \\ 
8 & 1 & 14 & 51 & 14 & 1& \\ 
9 & 1 & 28 & 169 & 169 & 28 & 1\\
10 & 1 & 42 & 483 & 884 & 483 & 42 & 1
\end{tabular}
      \caption{First few terms of the number of PADs in terms of number of excedances}
     \label{tab:exe}
\end{table}

\begin{table}[ht!]
\begin{multicols}{2}
\centering
\begin{tabular}{c|ccccccc}
\hline
 &  &  & $\mathfrak{d}^e_{n,k}$ &  &   & \\
 \hline
$n\setminus k$ & 2 & 3 & 4 & 5 & 6 & 7\\
\hline  
4 & 1 &  &  & &  &\\  
5 & 0 & 0 &  & & &  \\  
6 & 1 & 2 & 1 &  & &\\ 
7 & 0 & 3 & 3 & 0 & & \\ 
8 & 1 & 8 & 27 & 8 & 1& \\ 
9 & 0 & 13 & 83 & 83 & 13&0 \\
10 & 1 & 22 & 243 & 444 & 243 & 22 & 1
\end{tabular}
\centering
\begin{tabular}{c|ccccccc}
\hline
 &  &  & $\mathfrak{d}^o_{n,k}$ &  &   & \\
 \hline
$n\setminus k$ & 2 & 3 & 4 & 5 & 6 & 7\\
\hline  
4 & 0 &  &  & &  &\\  
5 & 1 & 1 &  & & &  \\  
6 & 0 & 0 & 0 &  & &\\ 
7 & 1 & 5 & 5 & 1 & & \\ 
8 & 0 & 6 & 24 & 6 & 0& \\ 
9 & 1 & 15 & 86 & 86 & 15&1 \\
10 & 0 & 20 & 240 & 440 & 240 & 20 & 0
\end{tabular}
     \end{multicols}
     \caption{First few values of $\mathfrak{d}_{n,k}$ in terms of their parity}
     \label{tab:exeboth}
\end{table}

\begin{pro}
The excedance distribution of a PAD is given by
\begin{align*}
\mathfrak{d}_{n,k}=\begin{cases}\sum_{i=1}^{k-1}d_{\lceil\frac{n}{2}\rceil,i}\,d_{\lfloor\frac{n}{2}\rfloor,k-i},\,\, \text{ if }\, 2\leq k\leq \lfloor\frac{n}{2}\rfloor\\
\sum_{i=1}^{n-k-1}d_{\lceil\frac{n}{2}\rceil,i}\,d_{\lfloor\frac{n}{2}\rfloor,n-k+i},\,\, \text{ if } \,\lfloor\frac{n}{2}\rfloor < k\leq n{-}2
\end{cases}.
\end{align*} 
\end{pro}

\begin{proof} 
To find the number of excedances of a PAD $\delta$ over $[n]$, we sum up the number of exeedances in $\delta_1$ and in $\delta_2$, where $(\delta_1,\,\delta_2)$ is the image of $\delta$ defined in $\Phi_n$. Since there are $d_{m,i}$  derangements in $D_m$ having $i$ excedances, for $i\in [1,\,\,m-1]$, the products $d_{m,i}\,d_{l,k-i}$, for $m=\lceil\frac{n}{2}\rceil$ and $l=\lfloor\frac{n}{2}\rfloor$, determine the number of PADs over $[n]$ having $k$ excedances. Summing up the products over the range $i=1,2,\ldots,k{-}1$ will give the first formula. The second formula follows from Proposition \ref{sym}.
\end{proof}

\begin{cor}\label{corolaryPAD}
 The excedance distribution of PADs in terms of their parity is given by:
\begin{align*}
\mathfrak{d}_{n,k}^e=\begin{cases}\sum_{i=1}^{k-1}(d^e_{\lceil\frac{n}{2}\rceil,i}\,d^e_{\lfloor\frac{n}{2}\rfloor,k-i}+d^o_{\lceil\frac{n}{2}\rceil,i}\,d^o_{\lfloor\frac{n}{2}\rfloor,k-i}),\,\, \text{ if }\, 2\leq k\leq \lfloor n/2\rfloor\\
\sum_{i=1}^{n-k-1}(d^e_{\lceil\frac{n}{2}\rceil,i}\,d^e_{\lfloor\frac{n}{2}\rfloor,n-k+i}+d^o_{\lceil\frac{n}{2}\rceil,i}\,d^o_{\lfloor\frac{n}{2}\rfloor,n-k+i}),\,\, \text{ if } \,\lfloor n/2\rfloor < k\leq n{-}2
\end{cases},\\
\mathfrak{d}_{n,k}^o=\begin{cases}\sum_{i=1}^{k-1}(d^e_{\lceil\frac{n}{2}\rceil,i}\,d^o_{\lfloor\frac{n}{2}\rfloor,k-i}+d^o_{\lceil\frac{n}{2}\rceil,i}\,d^e_{\lfloor\frac{n}{2}\rfloor,k-i}),\,\, \text{ if }\, 2\leq k\leq \lfloor n/2\rfloor\\
\sum_{i=1}^{n-k-1}(d^e_{\lceil\frac{n}{2}\rceil,i}\,d^o_{\lfloor\frac{n}{2}\rfloor,n-k+i}+d^o_{\lceil\frac{n}{2}\rceil,i}\,d^e_{\lfloor\frac{n}{2}\rfloor,n-k+i}),\,\, \text{ if } \,\lfloor n/2\rfloor < k\leq n{-}2
\end{cases}.
\end{align*}
\hfill$\square$
 \end{cor}
 An immediate consequence of this Corollary is
 \begin{pro}\label{diffexc formula}
 We have
 \begin{align*}
\mathfrak{f}_{n,k}=\mathfrak{d}_{n,k}^e-\mathfrak{d}_{n,k}^o=(-1)^n\,\max\{k-1,\, n{-}(k+1)\}
\end{align*}
for $n\geq 4$ and $2\leq k\leq n{-}2$.
 \end{pro}
 \begin{proof}
 Mantaci and Rakotondrajao (see \cite{mantaci2003exceedingly}) have proved the identity $d^o_{n,k}-d^e_{n,k}=(-1)^n$ using recursive argument. Applying this with the Corollary \ref{corolaryPAD}, we get the desired formula.
 \end{proof}
 
 \begin{table}[ht!]
    \centering
    \begin{tabular}{c|ccccccc}
\hline
 &  &  & $\mathfrak{f}_{n,k}$ &  &   & & \\
 \hline
$n\setminus k$ & 2 & 3 & 4 & 5 & 6 & 7 &\\
\hline  
4 & 1 &  &  & &  &  &\\  
5 & -1 & -1 &  & & &  &\\  
6 & 1 & 2 & 1 &  & & &\\ 
7 & -1 & -2 & -2 & -1 & & &\\ 
8 & 1 & 2 & 3 & 2 & 1 & &\\ 
9 & -1 & -2 & -3 & -3 & -2 & -1& \\
10 & 1& 2 & 3 & 4 & 3 & 2 & 1 
\end{tabular}
    \caption{The first few values of the difference $\mathfrak{d}^e_{n,k}-\mathfrak{d}^o_{n,k}$.}
    \label{tab:diff-exe}
\end{table}
 
 \begin{thm}
The exponential generating function for the sequence $\{\mathfrak{f}_{n,k}\}$ has the closed form
\begin{align*}
    \frac{1}{(1-u)^2}\left(u^2e^{-x}+e^{-ux}-2u\cosh{\sqrt{u}x}+\frac{u+u^2}{\sqrt{u}}\sinh{\sqrt{u}x}-(1-u)^2\right).
\end{align*}
\end{thm}

\begin{proof}
Let $\mathfrak{f}_n(u)=\sum_{k=2}^{n-2}\mathfrak{f}_{n,k}u^k$ and $\mathfrak{f}(x,u)=\sum_{n\geq 4} \mathfrak{f}_n(u)\frac{x^n}{n!}$. From Proposition \ref{diffexc formula}, we have
\begin{align*}
   \mathfrak{f}_{2m,k}&=\begin{cases}
    k-1,\quad \text{if }2\leq k\leq m\\
    2m-k-1,\quad \text{if }m< k\leq 2m{-}2
    \end{cases},\\
    \mathfrak{f}_{2m+1,k}&=\begin{cases}
    -(k-1),\quad \text{if }2\leq k\leq m\\
    -(2m-k),\quad \text{if }m < k\leq 2m{-}1
    \end{cases},
\end{align*}
for $m\geq 2$. So,
\begin{align*}
      \mathfrak{f}_{2m}(u)&=\sum_{k=2}^{m}(k-1)u^k+\sum_{k=m+1}^{2m-2}(2m-k-1)u^k=\frac{u^2-2u^{m+1}+u^{2m}}{(1-u)^2},\\
      \mathfrak{f}_{2m+1}(u)&=\sum_{k=2}^m-(k-1)u^k+\sum_{k=m+1}^{2m-1}-(2m-k)u^k=\frac{u^{m+2}+u^{m+1}-u^{2}-u^{2m+1}}{(1-u)^2},\\
      \mathfrak{f}(x,u)&=\sum_{m\geq 2}\mathfrak{f}_{2m}(u)\frac{x^{2m}}{(2m)!}+\sum_{m\geq 2}\mathfrak{f}_{2m+1}(u)\frac{x^{2m+1}}{(2m+1)!}\\
    &= \frac{1}{(1-u)^2}\left(u^2e^{-x}+e^{-ux}-2u\cosh{\sqrt{u}x}+\frac{u+u^2}{\sqrt{u}}\sinh{\sqrt{u}x}-(1-u)^2\right). 
    \end{align*}
  \end{proof}
 




\section*{Methodological remarks}
In this paper, most of our results are obtained in a way of splitting the permutations into two subwords. However, this method is not always applicable. One example is the number of PADs avoiding the pattern $p=1\,2$. The only derangement that avoid $p$ is $(1\,\,n\,)(2\,\,\,n{-}1\,)\cdots(\frac{n}{2}\,\,\,\frac{n+2}{2})$, for even $n$, that is, the derangement over $[n]$ with entries in decreasing order when written in linear representation. However, it does not exist if $n$ is odd, since $\frac{n+1}{2}$ is a fixed point. The PAD $\delta$ created from a pair $(\delta_1, \delta_2)$, by the mapping $\Phi_n^{-1}$, of two even length derangements that both avoids the pattern $p$ is $\delta=(1\,\,n{-}1\,)(3\,\,\,n{-}3\,)\cdots\big(\frac{n-2}{2}\,\,\,\frac{n+2}{2}\big)(2\,\,n\,)(4\,\,\,n{-}2\,)\cdots\big(\frac{n}{2}\,\,\,\frac{n+4}{2}\big)$, which is $n{-}1\,\,n\,\,n{-}3\,\,n{-}2\cdots 3\,\,4\,\,1\,\,2$ in linear form, has length $n\equiv 0 \pmod{4}$. However, each pair $i\,\,i{+}1$, where $i$ is an entry in odd position, is a subword with the occurrence of the pattern $p$ in $\delta$. This indicates that $\delta_1$ and $\delta_2$ avoid $p$ but $\delta$ doesn't. Things get even more complicated with patterns of length greater than 2.

\textbf{Final Remarks:}
As for now, we have not been successful in
finding the recurrence relations and generating functions for the sequences  $\{\mathfrak{d}_{n,k}\}_{n=0}^\infty$, $\{\mathfrak{d}_{n,k}^e\}_{n=0}^\infty$, and $\{\mathfrak{d}_{n,k}^o\}_{n=0}^\infty$. 

\section*{Acknowledgements}
The first author acknowledges the financial support extended by the cooperation agreement between International Science Program at Uppsala University and Addis Ababa University. Special thanks go to Prof.\ J\"orgen Backelin, Prof.\ Paul Vaderlind and Dr.\ Per Alexandersson of Stockholm University - Dept. of Mathematics, for all their valuable inputs and suggestions.
Many thanks to our colleagues from CoRS - Combinatorial Research Studio, for lively discussions and comments.

\bibliographystyle{plain}
\bibliography{main}
\end{document}